\documentclass[11pt]{amsart}
\usepackage{amsfonts,amsmath,amssymb,amsthm,cmtiup} 
\usepackage{mathrsfs}

\newtheorem{theorem}{Theorem}
\newtheorem*{theorem*}{Theorem}
 
\newtheorem*{lemma*}{Lemma} 
 
\newtheorem*{corollary*}{Corollary}

\newtheorem{fact}{Fact} 

\theoremstyle{definition} 

\newtheorem*{remark*}{Remark}
\newtheorem{definition}{Definition}
\newtheorem*{definition*}{Definition}

\newtheorem*{example*}{Example}

\newcommand{\ind}{\operatorname{ind}}
\newcommand{\Ind}{\operatorname{Ind}}

\newcommand{\ord}{\operatorname{ord}}

\def\cl#1{\skew3\overline{#1}}
\def\clx#1#2{\skew3\overline{#1}^{\raise.5pt\hbox{\scriptsize$\,#2$\!}}}

\def\clrx#1#2{\skew3\overline{#1\hspace{-0.1em}}^{\raise.5pt\hbox{\,\scriptsize$#2$}}}

\def\clax#1#2#3{\skew6\overline{#1_{\rlap{$\scriptstyle 
#2$}}}^{\raise.5pt\hbox{\hspace{0.2667em}\scriptsize$#3$}}\!
}
\def\clarx#1#2#3{\skew7\overline{\,#1_{\rlap{$\scriptstyle 
#2$}}}^{\raise.5pt\hbox{\hspace{0.2667em}\scriptsize$#3$}}\!
}
\def\Cl#1^#2{\cl{#1\kern-.09em}\kern.09em\vphantom{#1}^{#2}}

\newcommand\cU{\mathscr U}
\newcommand\cV{\mathscr V}

\def\a{\alpha}
\newcommand\bR{\mathbb R}
\newcommand\bN{\mathbb N}

\let\le\leqslant
\let\ge\geqslant

\begin{document}

\title{The Covering Dimension\\ 
of the Sorgenfrey Plane} 
\author{Ol'ga Sipacheva}

\begin{abstract}
It is proved that the square of the Sorgenfrey line has infinite covering dimension. 
\end{abstract}



\address{Department of General Topology and Geometry, Faculty of Mechanics and  Mathematics, 
M.~V.~Lomonosov Moscow State University, Leninskie Gory 1, Moscow, 199991 Russia}

\email{o-sipa@yandex.ru, osipa@gmail.com}

\maketitle

Covering dimension $\dim$ was originally introduced by Lebesgue for domains in Euclidean spaces~\cite{Lebesgue}. 
In 1933 \v Cech extended Lebesgue's definition to completely regular spaces~\cite{Cech}, and in 1950 Kat\v etov 
proposed a different definition, which coincided with \v Cech's one for normal spaces~\cite{Katetov}. Since then, 
these two versions of covering dimension have become equally popular. Thus, the books~\cite{AP}, \cite{Pears}, and 
\cite{Charalambous} use \v Cech's definition, while, e.g., \cite{Engelking} prefers Kat\v etov's one. There are 
also books on dimension theory which consider only normal or even only metrizable spaces~\cite{HW}, 
\cite{Nagami}.

We will denote \v Cech's covering dimension by $\dim$ and Kat\v etov's by $\dim_0$ and refer to them as 
``covering dimension'' and ``cozero covering dimension,'' respectively. The definitions of both of them involve 
the well-known notion of the order of a cover.

\begin{definition}
Let $X$ be a set, and let $\mathscr F$ be a family of its subsets. If 
there exists an integer $n\ge -1$ such that each point $x\in X$ belongs to at most $n+1$ elements of 
$\mathscr F$, then the least such $n$ is called the \emph{order} of $\mathscr F$ 
and denoted by \emph{$\ord \mathscr F$}. If there is no such an integer $n$, then 
$\ord \mathscr F=\infty$.
\end{definition}

In other words, $\ord \mathscr F$ is the largest integer $n$ for which there are $n+1$ members of $\mathscr F$ 
with nonempty intersection if such an integer exists and $\ord \mathscr F=\infty$ otherwise.

\begin{definition}[\v Cech (Kat\v etov)]
The \emph{covering dimension $\dim X$} (\emph{cozero covering dimension $\dim_0 X$}) 
of a topological space $X$ is the least integer $n\ge -1$ 
such that any finite open (cozero) cover of $X$ has a finite open (cozero) refinement of order $n$, provided that 
such an integer exists. If it does not exist, then $\dim X=\infty$ ($\dim_0 X = \infty$). 
\end{definition}

By a \emph{strongly zero-dimensional} space most authors mean a space $X$ with $\dim_0 X=0$ (although, e.g., 
Charalambous~\cite{Charalambous} uses this term for a space $X$ with $\dim X=0$, in which case a nonnormal space 
cannot be strongly zero-dimensional).  

The definitions of the two other most classical dimension functions, $\ind$ and $\Ind$, can be found in any book 
on dimension theory and are defined in more or less the same way. The only minor distinction in their 
definitions is that some authors consider the $\ind$ ($\Ind$) dimension to be undefined for 
nonregular (nonnormal) spaces, while others define these dimensions to be infinite for such spaces. We 
follow the latter. 

It follows directly from the definitions that 
\begin{itemize}
\item
if $\dim X = 0$, then $X$ is normal;
\item 
$\dim_0 X= \dim_0\beta X=\dim \beta X$.
\end{itemize}

This paper is devoted to the \v Cech covering dimension of the Sorgenfrey plane $S\times S$. (We use the 
standard notation $S$ for the Sorgenfrey line and assume that the underlying set of $S$ is $\bR$.) Obviously, 
$\ind S=\ind S^\omega=0$, $\Ind S\times S = \infty$ (because $S\times S$ is not normal), and $\dim S\times S> 0$ 
(for the same reason). It is also known that $\dim S=0$, because $S$ is Lindel\"of and $\dim \le \ind$ for 
Lindel\"of spaces (see, e.g., \cite[Proposition~5.3]{Charalambous}), and that $\dim_0 S^\omega=0$ and hence 
$\dim_0 S^n =0$ for all $n\in \omega$ \cite{Terasawa}. Moreover, $\dim_0 S^\kappa=0$ for any cardinal $\kappa$; 
Terasawa credited this result to K.~Morita in \cite{Terasawa}, and its proof can be found in~\cite{Wage}. We prove 
that $\dim S\times S=\infty$.

\begin{theorem}
$ \dim S\times S=\infty$. 
\end{theorem}

The proof of this theorem uses  shrinkings of  covers and several notions related to Baire categories. 

Recall that a cover $\{V_\a: \a \in A\}$ of a set $X$ is a \emph{shrinking} of a cover $\{U_\a: \a \in A\}$ if 
$V_\a\subset U_\a$ for each $\a\in A$. 

\begin{fact}[{see, e.g., \cite[Proposition~2.10]{Charalambous}}]
Given any space $X$ and any integer $n\ge -1$, 
$\dim X \le n$ if and only if any finite open cover 
$\{U_1,\dots, U_k\}$ of $X$ has an open shrinking $\{V_1,\dots, V_k\}$ of order $\le n$.
\end{fact}

Sets of first Baire category are said to be \emph{meager} and sets of the second Baire category, 
\emph{nonmeager}. Complements to meager sets are \emph{comeager}, or \emph{residual}. It is easy to show 
that any comeager set is nonmeager, but not vice versa. 

An example of a nonmeager set in $\mathbb R$ which is not comeager is any \emph{Bernstein set}, 
i.e., a set $B\subset \bR$ such that both $B$ and $\bR\setminus B$ intersect all uncountable 
closed (or, equivalently, compact) subsets of~$\bR$.

\begin{fact}[\cite{SL}]
\label{Fact1}
There exist $ 2^{\aleph_0}$ disjoint Bernstein sets.
\end{fact}

\begin{fact}[{see, e.g., \cite[Chap.~3, Sec.~40, I, Theorem~3]{Kuratowski}}] 
\label{Fact2}
The intersection of any Bernstein set with any nontrivial interval is of second category in~$\bR$. 
\end{fact}

\begin{proof}[Proof of the theorem]
Take any $n\in \bN$. We must prove that $\dim S\times S> n$. To this end, it suffices to produce a finite open 
cover of $S\times S$ which has no finite open shrinking of order $\le n$. We take disjoint Bernstein sets $B_1, 
\dots, B_{n+1}$ and put $B_{n+2}=\bR\setminus (B_1\cup\dots \cup B_{n+1})$; clearly, $B_{n+2}$ is a Bernstein 
set as well. 

We set $D = \{(x,-x): x\in \bR\}\subset S\times S$ and consider the cover 
$\cU=\{U_0,U_1,\dots, U_{n+2}\}$ of $S\times S$, where 
$$
U_0=\{(x, y): y<-x\}
$$
and
$$
U_i = \bigcup \{[x,\infty)\times [-x,\infty): x\in B_i\}\quad \text{for \ 
$i\le n+2$.}
$$
Let $\cV=\{V_0, V_1, \dots, V_{n+2}\}$ be an open shrinking of $\cU$. 

Obviously, $V_0=U_0$ and $V_i\cap D=U_i\cap D=\{(x,-x):x\in B_i\}$ for $i= 1,\dots, n+2$. 

For $i =1, \dots, n+2$ and  $k\in \bN$, we set $B_i(k)=\{x\in B_i: [x,x+\frac 
1k)\times [-x,-x+\frac 1k)\subset V_i\}$. Obviously, $B_i(k)\subset B_i(m)$ for $m\le k$. 
Our immediate goal is to 
define positive integers $k_1\le \dots\le k_{n+2}$ and nonempty intervals $ (a_1, b_1)\supset \dots \supset 
(a_{n+2}, b_{n+2})\ne \varnothing$ so that the set  $B_i (k_j)\cap (a_j, b_j)$ is dense in $(a_j, b_j)$ with 
respect to the usual topology on $\bR$ for each $j\le n+2$ and each $i \le j$.

Since $\bigcup\limits_{i\in \bN}B_1(i)= B_1$ and $B_1 $, being a Bernstein set, is nonmeager in 
$\bR$, it follows that $B_1(k_1)$ is not nowhere 
dense in $\bR$ for some positive integer $k_1$. This means that 
$B_1(k_1)\cap U$ is dense in $U$ for some nonempty open set $U\subset \bR$; hence there exist $a_1, 
b_1\in \bR$, $a_1< b_1$, for which $\cl{B_1(k_1)\cap (a_1, b_1)}=[a_1,b_1]$ (the closure is in $\bR$). 

Suppose that $1<j\le n+2$ and we have already found integers $k_ i$ and constructed intervals 
$(a_i, b_i)$ for $i<j$. 

The set $B_j$ intersects all uncountable closed subsets of the interval $[ a_{j-1}, b_{j-1}]$, because all such 
sets are also closed in $\bR$, and 
$B_j\cap [ a_{j-1}, b_{j-1}]$ is of second Baire category in $[a_{j-1}, b_{j-1}]$ (see Fact~\ref{Fact2}). 
Since $B_j(k)\subset B_j(m)$ for $k\le m$, there exists a positive integer $k_j\ge k_{j-1}$ for which $B_j(k)$ is 
not nowhere dense in $[a_{j-1}, b_{j-1}]$. Let $(a_j, b_i)\subset (a_{j-1}, b_{j-1})$ be a nonempty interval for 
which $\cl{B_j(k_j)\cap (a_j,b_j)}=[a_j,b_j]$. 

At the $(n+2)$th step we obtain an interval $(a_{n+2}, b_{n+2})$ and a number $k_{n+2}$. We set  $(a,b)=(a_{n+2}, 
b_{n+2})$ and $k= k_{n+2}$. By construction $(a,b)\subset (a_i, b_i)$ and  
$\cl{B_i(k_i)\cap (a,b)}=[a,b]$ for all $i \le n+2$. Since $k\ge k_i$ and $B_i(k_i)\subset B_i(k)$ for all 
$i\le n+2$, it follows that $\cl{B_i(k)\cap (a,b)}=[a,b]$ for all $i \le n+2$. 

For each $j\le n+2$, there is a point $x_j\in (a,b)\cap (a,a+\frac 1k)\cap B_j(k)$. It is easy to see that  
$[x_j,x_j+\frac 1k)\times  [-x_j,-x_j+\frac 1k)\ni (a+\frac 1k, -a)$. By the definition of $B_j(k)$ we have 
$[x_j,x_j+\frac 1k)\times  [-x_j,-x_j+\frac 1k)\subset V_j$. Therefore, $(a+\frac 1k, -a)\in V_j$ for all 
$j=1,\dots, n+2$, and the order of the cover $\cV$ equals~$n+1$. 
\end{proof}

\begin{remark*}
The same argument proves that the covering dimension of the Niemytski plane is infinite.
\end{remark*}

\end{document}